\documentclass[12pt,reqno]{amsart}

\usepackage[dvips]{graphicx} 

\usepackage{hyperref}

\usepackage{breakurl}   

\usepackage{amssymb,latexsym, amsmath, amscd, array, hyperref, epigraph}

\usepackage{tikz}
\usetikzlibrary{decorations.pathreplacing}

\newtheorem{theorem}{Theorem}[section]
\newtheorem{lemma}[theorem]{Lemma}

\newtheorem{corollary}[theorem]{Corollary}

\theoremstyle{definition}
\newtheorem{definition}[theorem]{Definition}
\newtheorem{example}[theorem]{Example}

\newtheorem{remark}[theorem]{Remark}

\numberwithin{equation}{section}

\newcommand{\pmba}{\raisebox{-6pt}{\begin{tikzpicture}[scale=.3]
\draw [dashed, decorate, decoration={brace, amplitude=5pt}] (0,0) -- (0,2);
\end{tikzpicture}}}

\newcommand{\pmbb}{\raisebox{-6pt}{\begin{tikzpicture}[scale=.3]
\draw [dashed, decorate, decoration={brace, amplitude=5pt}] (2,2) -- (2,0);
\end{tikzpicture}}}

\newcommand{\pmbaa}{\pmba\hspace{-5pt}\pmba}
\newcommand{\pmbbb}{\pmbb\hspace{-5pt}\pmbb}
\newcommand\N {{\mathbb N}} 
\newcommand\R {{\mathbb R}}

\newcommand\Z {{\mathbb Z}} 
\newcommand\st{{\rm st}} 
\newcommand\astx{{}^{\ast}\hspace*{-2.9pt}X}
\newcommand\astu{{}^{\ast}\hspace*{-1.14pt}U}
\newcommand\asta{{{}^\ast\hspace*{-3.5pt}A}}
\newcommand\asts{{}^{\ast}\hspace*{-1.7pt}S}

\newcommand{\astb}{{}^{\ast}\hspace*{-2.5pt}B}
\newcommand\astf{{{}^{\ast}\hspace*{-3pt}f}}
\newcommand\astn{{{}^{\ast}\hspace*{-1pt}\N}}
\newcommand\astp{{{}^{\ast}\hspace*{-0.5pt}\mathbb{P}}}
\newcommand\astr{{{}^\ast\hspace*{-.6pt}\R}}
\newcommand\asto{{{}^\ast\hspace*{-.6pt}O}}
\newcommand\eqi{\longleftrightarrow}
\newcommand\ts{T}
\newcommand\astt{{}^\ast T}
\DeclareMathOperator\sh{sh}

\title{Effective infinitesimals in~$\mathbb R$}

\begin{document}

\author{Karel Hrbacek} \address{K.~Hrbacek, Department of Mathematics,
  City College of CUNY, New York, NY 10031}
\email{khrbacek@icloud.com}

\author{Mikhail G. Katz}\address{M.~Katz, Department of Mathematics,
  Bar Ilan University, Ramat Gan 5290002
  Israel}\email{katzmik@math.biu.ac.il}


 
\begin{abstract}
We survey the effective foundations for analysis with infinitesimals
developed by Hrbacek and Katz in 2021, and detail some applications.
Theories SPOT and SCOT are conservative over respectively ZF and
ZF+ADC.  The range of applications of these theories illustrates the
fact that analysis with infinitesimals requires no more choice than
traditional analysis.  The theory SCOT incorporates in particular all
the axioms of Nelson's Radically Elementary Probability Theory, which
is therefore conservative over ZF+ADC.
\end{abstract}

\thispagestyle{empty}

\maketitle
\tableofcontents

\section{Introduction}
\label{s1}

Let ZF be the Zermelo--Fraenkel set theory.  Let ZFC be the
Zermelo--Fraenkel set theory with the axiom of choice.  Let ACC be the
axiom of countable choice, and ADC the axiom of (countable) dependent
choice.  The theories ZF, ZF+ACC, and ZF+ADC have the advantage (over
ZFC) of not entailing set-theoretic paradoxes such as Banach--Tarski.
Similarly, ZF, ZF+ACC, and ZF+ADC do not prove the existence of
nonprincipal ultrafilters.

The theories SPOT and SCOT developed in \cite{21e} provide frameworks
for analysis with infinitesimals that are conservative respectively
over ZF and ZF+ADC, and therefore share the same advantage (the axioms
of SPOT and SCOT appear in Section~\ref{s4}).  Mathematicians
generally consider theorems provable in ZF as more effective than
results that require the full ZFC for their proof, and many feel this
way not only about ZF but about ZF+ADC, as well.  In this sense, the
theories SPOT and SCOT enable an effective development of analysis
based on infinitesimals.  Some applications were already presented
in~\cite{21e}, such as (local) Peano's existence theorem for
first-order differential equations \cite[Example 3.5]{21e} and
infinitesimal construction of Lebesgue measure via counting measures
\cite[Example~3.6]{21e}.

We first consider the case of compactness.  In Section~\ref{s2}, we
present the traditional extension view.  Following an outline of SPOT
and SCOT in Sections~\ref{ist} and \ref{s4}, we deal with compactness
in internal set theories in Section~\ref{s3}.  After preliminaries on
continuity in Section~\ref{s6}, we present an effective proof using
infinitesimals of the compactness of a continuous image of a compact
set in Section \ref{s7}.  After preliminaries on uniform continuity in
Sections~\ref{s8} and \ref{s9}, we present an effective proof using
infinitesimals of the Heine--Borel theorem in Section~\ref{s10}.  In
Section~\ref{s11}, we show that Nelson's \emph{Radically Elementary
  Probability Theory} is a subtheory of SCOT.

\section{Compactness in the extension view}
\label{s2}

In this section, we analyze compactness from the viewpoint of
traditional extensions~$\R\hookrightarrow\astr$ to hyperreals.  These
cannot be constructed in ZF+ADC and cannot be described as effective
in the sense of Section~\ref{s1}.  In Section~\ref{s3}, we will
present an effective treatment of compactness in axiomatic frameworks
for analysis with infinitesimals.

For~$\N$,~$\R$,~$\mathbb P=\mathcal{P}(\R)$, or any set~$X$, the
corresponding nonstandard extensions~$\astx$, etc.~ satisfying the
transfer principle can be formed either via the compactness theorem of
first-order logic, or via ultrapowers~$X^\N/\mathcal F$, etc., in
terms of a fixed nonprincipal ultrafilter~$\mathcal F$.

\begin{lemma}
\label{l1012}
For a finite union, the star of the union is the union of stars.
\end{lemma}

\begin{proof}
Given sets~$A,B\subseteq X$, we have
\begin{equation}
\label{e11}
(\forall{}y\in{}X)\big[y\in{}A\cup{}B \eqi
  (y\in{}A)\vee(y\in{}B)\big].
\end{equation}
Applying upward transfer to \eqref{e11}, we obtain
\[
(\forall{}y\in{}\astx) \big[y\in{}^\ast\!(A\cup B) \eqi
  (y\in\asta)\vee(y\in\astb)\big],
\]
and the claim follows by induction.
\end{proof}

\begin{theorem}
\label{t971}
If~$\langle A_n \colon n\in\N\rangle$ is a nested sequence of nonempty
subsets of~$\R$ then the sequence~$\langle\asta_n\colon n\in\N\rangle$
(standard~$n$) has a common point.
\end{theorem}

\begin{proof}
Let~$\mathbb {P}=\mathcal{P}(\R)$ be the set of all subsets of~$\R$.
Consider a sequence~$\langle A_n\in\mathbb{P}\colon n\in\N\rangle$
viewed as a function~$f\colon\N\to\mathbb P,\; n\mapsto A_n$.  By the
extension principle we have a function~$\astf\colon \astn\to \astp$.
Let~$B_n=\astf(n)$.  For each standard~$n$, we
have~$B_n=\asta_n\in\astp$.%
\footnote{The injective map~$\ast\colon\mathbb{P}\to\astp$ sends~$A_n$
  to~$\asta_n$.  For each standard natural~$n$ we have a symbol~$a_n$
  in the appropriate language (including at least the names for all
  subsets of~$\R$),
%
%
whose standard interpretation is~$A_n\in\mathbb P$.  Meanwhile the
nonstandard interpretation of~$a_n$ is the entity~$\asta_n\in\astp$.
The sequence~$\langle A_n \colon n\in\N\rangle$ in~$\mathbb P$ is the
standard interpretation of the symbol~$a=\langle a_n\rangle$.
Meanwhile, the nonstandard interpretation of the symbol~$a$ is
$\langle B_n \colon n\in\astn\rangle$ in~$\astp$.  In particular, one
has~$B_n=\asta_n$ for standard~$n$.}
For a nonstandard value of the index~$n=H$, the entity~$B_H\in\astp$
is by definition internal but is in general not the natural extension
of any subset of~$\R$.

If~$\langle A_n \rangle$ is a \emph{nested} decreasing sequence
in~$\mathbb{P}\setminus\{\varnothing\}$ then by transfer~$\langle
B_n:n\in\astn \rangle$ is nested in~$\astp\setminus\{\varnothing\}$.
Let~$H$ be a fixed nonstandard index.  Since~$n<H$ for each
standard~$n$, the set~$\asta_n\subseteq\astr$ includes~$B_H$.  Choose
any element~$c\in B_H$.  Then~$c$ is contained in~$\asta_n$ for each
standard~$n$:
\[
c\in \bigcap_{n\in\N} \asta_n
\]
as required.%
\footnote{The conclusion of non-empty intersection remains valid for
  any nested sequence of nonempty \emph{internal} sets, i.e., members
  of~$\astp$; see e.g., \cite[Theorem~11.10.1, p.\,138]{Go98} (the
  proof is more involved).}
\end{proof}

\begin{definition}
\label{d175}
Let~$I$ be a set.  A collection~$\mathcal{H}\subseteq \mathcal{P}(I)$
has the \emph{finite intersection property} if the intersection of
every nonempty finite subcollection of~$\mathcal{H}$ is nonempty,
i.e.,
\[
B_1 \cap\cdots\cap B_n \not=\varnothing \text{ for all } n\in\N \text{
  and all } B_1 , \ldots , B_n \in \mathcal{H}.
\]
\end{definition}

Then Theorem~\ref{t971} has the following equivalent formulation.

\begin{corollary}
[Countable Saturation]
\label{r1013}
If a family of subsets~$\{A_n\}_{n\in\N}$ has the finite intersection
property (see Definition~\textnormal{\ref{d175}}) then the
intersection 
$\bigcap_{n\in\N} \asta_n$ is nonempty.
\end{corollary}

Recall that a topological space~$\ts$ is second-countable if its
topology admits a countable base.  Recall that a space is Lindel\"of
if every open cover includes a countable subcover.  A second countable
space is necessarily Lindel\"of (over ZF+ACC).  If~$\ts$ is a
separable metric space then~$\ts$ is second countable and hence
Lindel\"of.

A point~$y\in\astt$ is called \emph{nearstandard in~$\ts$} if~$y$ is
infinitely close to a standard point~$p\in \ts$, i.e., such that~$y$
is contained in the star~$\astu$ of every open neighborhood~$U$
of~$p$.  The intersection of all such~$\astu$ is called the
\emph{halo} of~$p$.  The relation~$x\simeq y$ holds if and only if for
all open sets~$O$,~$x\in\asto$ if and only if~$y\in\asto$.

\begin{theorem}
\label{t26}
Assume~$\ts$ is Lindel\"of.  Then the following two conditions are
equivalent:
\begin{enumerate}
\item
$\ts$ is compact (i.e., every open cover admits a finite subcover);
\item
every~$y\in\astt$ is nearstandard in~$\ts$.
\end{enumerate}
\end{theorem}

\begin{proof}[Proof of~$(1)\Rightarrow(2)$]
Assume~$\ts$ is compact, and let~$y\in\astt$.  Let us show that~$y$ is
nearstandard in~$\ts$ (this direction does not require saturation).

Suppose on the contrary that~$y$ is not nearstandard in~$\ts$,
i.e.,~$y$ is not in the halo of any (standard) point~$p\in \ts$.  Then
we can form the open cover~$\mathcal{U}$ of~$T$ containing all open
sets~$U$ such that
\begin{equation}
\label{e1111}
y\notin \astu.  
\end{equation}
Since~$\ts$ is compact,~$\mathcal U$ includes a finite subcover
$U_{1},\ldots,U_{n}$.  Applying Lemma~\ref{l1012} to the finite
union~$\ts=U_{1}\cup\cdots\cup{}U_{n}$, we obtain
\[
\astt=\astu_{1} \cup\cdots\cup{}\astu_{n}.
\]
Hence~$y$ is in one of the~$\astu_{i}$,~$i=1,\ldots,n$,
contradicting~\eqref{e1111}.  The contradiction establishes that~$y$
is necessarily nearstandard in~$\ts$.
\end{proof}

\begin{proof}[Proof of~$(2)\Rightarrow(1)$]
This direction exploits saturation.  Assume each
$y\in\astt$ is nearstandard in~$\ts$.  Given an open cover~$\{U_a\}$
of~$\ts$, we need to find a finite subcover.  Since~$\ts$ is
Lindel\"of, we can assume that the cover is countable.

Suppose on the contrary that no finite subcollection of~$\{U_a\}$
covers~$\ts$.  Then the complements~$S_a$ of~$U_a$ form a countable
collection of (closed) sets~$\{S_a\}$ with the finite intersection
property.  Applying countable saturation (Corollary~\ref{r1013}) to
this countable family, we conclude that the intersection of
all~$\asts_a$ is non-empty.  Let~$y\in \bigcap_a \asts_a$.  By
assumption, there is a point~$p \in \ts$ such that
\begin{equation}
\label{e1112}
y \simeq p.
\end{equation}

Since~$\{U_a\}$ is a cover of~$\ts$, it contains a set~$U_b$ such
that~$p\in{}U_b$, and hence~$y\in \astu_b$ since~$U$ is open.
But~$y\in\asts_a$ for all~$a$, in particular~$y\in\asts_b$,
so~$y\not\in\astu_b$ by Lemma~\ref{l1012},
contradicting~\eqref{e1112}.  The contradiction establishes the
existence of a finite subcover.
\end{proof}

\section{Internal set theories}
\label{ist}

In this section we explain in what sense analysis with infinitesimals
does not require the axiom of choice any more than traditional
non-infinitesimal analysis, following \cite{21e}.  There are two
popular approaches to Robinson's nonstandard mathematics (including
analysis with infinitesimals):
\begin{enumerate}
\item
model-theoretic, and
\item
axiomatic/syntactic.
\end{enumerate}
For a survey of the various approaches see \cite{17f}.

The model-theoretic approach (including the construction of the
ultrapower) typically relies on strong forms of the axiom of
choice. The axiomatic/syntactic approach turns out to be more
economical in the use of foundational material, and exploits a richer
\st-$\in$-language, as explained below.

The traditional set-theoretic foundation for mathematics is
Zermelo--Fraenkel set theory (ZF).  The theory ZF is a set theory
formulated in the~$\in$-language.  Here ``$\in$" is the two-place
membership relation.  In ZF, all mathematical objects are built up
step-by-step starting from~$\emptyset$ and exploiting the one and only
relation~$\in$.

For instance, the inequality~$0<1$ is formalized as the membership
relation~$\emptyset\in\{\emptyset\}$, the inequality~$1<2$ is
formalized as the membership relation
$\{\emptyset\}\in\{\emptyset,\{\emptyset \}\}$, etc. Eventually ZF
enables the construction of the set of natural numbers~$\N$, the ring
of integers~$\Z$, the field of real numbers~$\R$, etc.

For the purposes of mathematical analysis, a set theory SPOT has been
developed in the more versatile \st-$\in$-language (its axioms are
given in Section~\ref{s4}).  Such a language exploits a predicate
{\st} in addition to the relation~$\in$.  Here ``\st" is the one-place
predicate \textbf{standard} so that \st($x$) is read ``$x$ is
standard".

\begin{theorem}
[\cite{21e}]
\label{conservative}
The theory {\rm SPOT} is a conservative extension of~{\rm ZF}.
\end{theorem}

This means that every statement in the~$\in$-language provable in SPOT
is provable already in ZF.  In particular, the axiom of choice and the
existence of non-principal ultrafilters are not provable is SPOT,
because they are not provable in ZF.  Thus SPOT does not require any
additional foundational commitments beyond ZF.

\begin{remark}
The Separation Axiom of ZF asserts, roughly, that for any
$\in$-formula~$\phi$ and any set~$A$, there exists a set~$S$ such
that~$x\in S$ if and only if~$x \in A \,\wedge\, \phi(x)$ is true.
This remains valid in SPOT which is a conservative extension of ZF.\,
But Separation does not apply to formulas involving the new predicate
\st.  Specifically, Separation does not apply to the predicate {\st}
itself.
\end{remark}

\begin{example}
The collection of standard natural numbers is not a set that could be
described as~``$\{x\in\N:\st(x)\}$.''  Such external collections can
be viewed informally as classes defined by the corresponding
predicate.  Thus, in \cite{21e} one uses the \emph{dashed curly brace}
notation~$\pmbaa n\in\N:\st(n)\pmbbb$ for such a class, when
convenient.  Writing~$k\in\pmbaa n\in\N:\st(n)\pmbbb$ is equivalent to
writing ``$\st(k)$ (is true)''.  In many cases the passage from a
predicate to a set turns out to be unnecessary: as mentioned in the
introduction, in SCOT (conservative over ZF+ADC) one can give an
infinitesimal construction of the Lebesgue measure; in BST (a
modification of Nelson's IST, possessing better meta-mathematical
properties), the Loeb measure can be handled, as well; see \cite{23b}.
\end{example}

\begin{remark}[Sources in Leibniz]
\label{r34}
The predicate {\st} formalizes the distinction already found in
Leibniz between assignable and inassign\-able numbers.  An
inassignable (nonstandard) natural number~$\mu$ 
%
%
is greater than every assignable (standard) natural number.  One of
the formulations of Leibniz's \emph{Law of Continuity} posits that
``the rules of the finite are found to succeed in the infinite and
vice versa'' (cf.\;Robinson \cite[p.\;266]{Ro66}), formalized by
Robinson's transfer principle.  See further in \cite{21a}, \cite{22a},
and \cite{23c}.
\end{remark}

If~$\mu\in\N$ is a nonstandard integer, then its
reciprocal~$\varepsilon=\frac1\mu\in\R$ is a positive infinitesimal
(smaller than every positive standard real). Such an~$\varepsilon$ is
a nonstandard real number.

A real number smaller in absolute value than some standard real number
is called \emph{limited}, and otherwise \emph{unlimited}.  SPOT proves
that every nonstandard natural number is unlimited
\cite[Lemma~2.1]{21e}.

The theory SPOT enables one to take the standard part, or shadow, of
every limited real number~$r$, denoted~$\sh(r)$.  This means that the
difference~$r-\sh(r)$ is infinitesimal.

The derivative of the standard function~$f(x)$ is then\,
$\sh\big(\frac{f(x+\varepsilon)-f(x)}{\varepsilon}\big)$ for nonzero
infinitesimal~$\varepsilon$.  In more detail, we have the following.

\begin{definition}
Let~$f$ be a standard function, and~$x$ a standard point.  A standard
number~$L$ is the \emph{slope} of~$f$ at~$x$ if
\begin{equation}
\label{e31}
(\forall^{in}\varepsilon)(\exists^{in}\lambda)\,
f(x+\varepsilon)-f(x)=(L+\lambda)\varepsilon.
\end{equation}
where~$\forall^{in}$ and~$\exists^{in}$ denote quantification over
infinitesimals.%
\footnote{For further details, see note~\ref{f3}.}
\end{definition}

The Riemann integral of~$f$ over~$[a,b]$ (with~$f, a, b$ standard), when
it exists, is the shadow of the sum~$\sum_{i=1}^\mu f
(x_i)\varepsilon$ as~$i$ runs from 1 to~$\mu$, where the~$x_i$ are the
partition points of an equal partition of~$[a,b]$ into~$\mu$
subintervals.  For a fuller treatment see \cite[Example~2.8]{21e}.

The (external) relation of infinite proximity~$x\simeq y$
for~$x,y\in\R$ is defined by requiring~$x-y$ to be infinitesimal.

\section{The axioms of SPOT and SCOT}
\label{s4}

We will now present the axioms that enable this effective approach
(conservative over ZF) to analysis with infinitesimals.

\subsection{Axioms of the theory SPOT}

SPOT is a subtheory of axiomatic (syntactic) theories developed in the
1970s independently by Hrbacek \cite{Hr78} and Nelson \cite{nelson}.
In addition to the axioms of ZF, SPOT has three axioms: Standard
Part, Nontriviality, and Transfer (for the historical origins of the
latter see Remark~\ref{r34}):

\begin{quote}
T (Transfer) Let~$\phi$ be an~$\in$-formula with standard parameters.
Then~$\forall^\st x \,\phi (x) \to \forall x \, \phi(x)$.
\end{quote}

\begin{quote}
O (Nontriviality)~$\exists\nu\in\N\,\forall^\st n\in\N\,(n\not=\nu)$.
\end{quote}

\begin{quote}
SP (Standard Part) Every limited real is infinitely close to a
standard real.
\end{quote}

An equivalent existential version of the Transfer axiom is~$\exists
x\; \phi(x) \implies \exists^{\st} x\; \phi(x)$, for~$\in$-formulas
$\phi~$ with standard parameters.

Nontriviality asserts simply that there exists a nonstandard integer.

An equivalent version of Standard Part is the following.

\begin{quote}
SP$'$ (Standard Part)
\end{quote}
\[
\forall A\subseteq\N \, \exists^\st B\subseteq\N \, \forall^\st n\in\N
\, (n\in B\leftrightarrow n\in A).
\]

\begin{remark}
The latter formulation can be motivated intuitively as follows.  Given
a real number~$0<r<1$, consider its base-2 decimal expansion.  Let~$A$
be the set of ranks where digit~$1$ appears.  The set~$A$ is not
standard if~$r$ is not standard.  The corresponding standard set~$B$
(whose existence is postulated by SP$'$) can be thought of as the set
of nonzero digits of the shadow~$\sh(r)$ of~$r$.  The fact that~$r$
and~$\sh(r)$ are infinitely close reflects the fact that~$A$ and~$B$
agree at all limited ranks.  The detailed argument is a bit more
technical because binary representation (like decimal representation)
is not unique; see \cite[Lemma~2.4]{21e}.
\end{remark}

In the model-theoretic frameworks one has three categories of sets:
sets that are natural extensions of, say, subsets of~$\R$, more
general internal sets, as well as external sets.  In the axiomatic
frameworks, the standard and nonstandard sets correspond to the
natural extensions and the internal sets, whereas there are no
external sets.

\subsection{Additional principles}

Recall that SPOT proves that standard integers are an initial segment
of~$\N$ \cite[Lemma~2.1]{21e}.

\begin{lemma} 
[Countable Idealisation] 
\label{countideal} 
Let~$\phi$ be an~$\in$-formula with arbitrary parameters.  The theory
{\rm SPOT} proves the following:
\[
\forall^\st n \in \N\;\exists x\; \forall m \in \N\; (m \le n
\rightarrow \phi(m,x)) \; \eqi \; \exists x \; \forall^\st n \in \N \;
\phi(n,x).
\]
\end{lemma}

This is proved in \cite[Lemma~2.2]{21e}.  One could elucidate
Countable Idealization by means of an equivalent version with
countable~$A$ in words as follows.  If for every standard finite
subset~$a \subseteq A$ there is some~$x$ such that for all~$z \in a$,
one has $\phi(z,x)$, then there is a single~$x$ such that~$\phi(z,x)$
holds for all standard~$z \in A$ simultaneously (the converse is
obvious given that all elements of a standard finite set are standard,
which is a consequence of \cite[Lemma 2.1]{21e}).  This is analogous
to saturation (see Corollary~\ref{r1013}).

\begin{definition}  
SN is the standardisation principle for \st-$\in$-formulas with no
parameters.  Namely, let~$\phi(v)$ be an~$\st$-$\in$-formula with no
parameters.  Then
\begin{equation}
\label{e41}
\forall^\st\!A\;\exists^\st S\;\forall^\st x\;(x\in S \eqi x \in
A\,\wedge\,\phi(x)).
\end{equation}
\end{definition}

It is proved in \cite[Lemma~6.1]{21e} that SN is equivalent to
standardisation for formulas with only standard parameters.  

Although separation does not hold, SN is a kind of
\emph{approximation} to it in the following sense.  The standard
elements of~$S$ (but not all elements) are exactly those for which
$\phi(x)$ holds.  Note also that the assumption that all parameters
are standard is necessary to maintain conservativity over ZF, because
otherwise one could prove the existence of nonprincipal ultrafilters
(see \cite{21e}).

Note that SPOT+SN is also conservative over ZF \cite[Theorem B,
  p.\;4]{21e}.  The axiom SN enables one to give a simple
infinitesimal definition of the derivative function conservatively
over ZF.%
\footnote{\label{f3}To dot the i's, let~$\phi_f(x,L)$ be the formula
  of \eqref{e31} depending on the standard parameter~$f$, a
  real-valued function.  Let~$A=\R^2$ in \eqref{e41}.  Then passing
  from~$f$ to ~$f'$ is enabled by the following consequence of
  \eqref{e41} containing only standard parameters:\;~$\exists^\st f'
  \; \forall^\st (x,L) \; \big( (x,L)\in f' \eqi (x,L)\in \R^2 \wedge
  \phi_f(x,L) \big)$\; where~$f'$ is thought of as its graph in the
  plane.}

SCOT incorporates the following choice-type axiom CC (which is a
strengthening of SP); see \cite[Section 3, p.\,10]{21e}.

\begin{definition}
\label{d44}
(CC) Let~$\phi (u,v)$ be an~$\st$-$\in$-formula with arbitrary
parameters. Then
\[
\forall^{\st} n \in \N\; \exists x\; \phi(n,x) \;\longrightarrow\;
\exists f\, (f \text{ is a function} \,\wedge\, \forall^{\st} n \in
\N\; \phi(n, f(n)).
\]
\end{definition}

The following definition was given in \cite[p.\,10]{21e}.

\begin{definition}
SCOT is the theory SPOT+ADC+SN+CC.
\end{definition}

SCOT (in fact, its subtheory SPOT+CC) also proves the following
statement SC \cite[Lemma~3.1]{21e}.

\begin{definition}
\label{d45}
SC (Countable Standardisation) Let~$\psi(v)$ be an~$\st$-$\in$-formula
with arbitrary parameters.  Then
\[
\exists^{\st} S\; \forall^{\st} n \; (n \in S \eqi n \in \N
\,\wedge\, \psi(n)).
\]
\end{definition}

\section{Compactness in internal set theories}
\label{s3}

In this section, we use infinitesimals to deal with compactness
conservatively over ZF or ZF+ADC, as indicated below (the traditional
extension view was already elaborated in Section~\ref{s2}).

Let~$\ts$ be a standard topological space.  A point~$x \in \ts$ is
\emph{nearstandard in~$\ts$} if there is a standard~$p\in \ts$ such
that~$p \in O$ implies~$x \in O$ for every standard open set~$O$ (in
other words,~$x$ is in the \emph{halo} of~$p$.)

\begin{lemma}
\label{l31}
Assume~$\ts$ is a standard Lindel\"of space.  If every~$x$ in~$\ts$ is nearstandard in
$\ts$ then~$\ts$ is compact.
\end{lemma}

\begin{proof}
Suppose~$\ts$ is not compact.  By downward transfer, there is a
standard countable cover~$\mathcal U$ of~$\ts$ by open sets such that
for every (standard) finite~$k$-tuple~$O_1,\ldots,O_k\in\mathcal U$
there is a~$p\in \ts\setminus\bigcup_{1\leq i\leq k}O_i$.
%
%
By Countable Idealisation with the standard
parameters~$\ts$,~$\mathcal U$, there is~$x \in \ts$ such that~$x
\not\in O$ for any standard~$O \in \mathcal U$.  Such an~$x$ is not
nearstandard in~$\ts$, because if~$x$ were in the halo of some
standard~$p \in \ts$, we would have a standard~$O \in \mathcal U$ such
that~$p \in O$ ($\mathcal U$ is a cover) and hence~$x \in O$, a
contradiction.%
\footnote{An analogous proof goes through for arbitrary standard
  topological spaces if one has \emph{full} idealisation (with
  standard parameters) such as in the theory BSPT$'$ \cite{21e}, which
  is still conservative over ZF (unfortunately it is not known whether
  SN can be added to it conservatively over ZF).  Note that BSPT$'$
  proves the existence of a finite set containing all standard reals
  \cite[p.\,10]{21e}.  Such sets are used in Benci's approach to
  measure theory.}
%
%
%
\end{proof}

\begin{lemma}
\label{l32}
Assume~$\ts$ is a standard second countable space.  If~$\ts$ is compact then every~$x
\in \ts$ is nearstandard in~$\ts$.
\end{lemma}

\begin{proof}
Suppose~$\mu \in \ts$ is not nearstandard in~$\ts$.  Then for every
standard~$p\in \ts$ there is a standard open set~$O$ such that~$p \in
O$ and~$\mu\not\in O$.  Let~$\mathcal B$ be a standard countable base
for the topology of~$\ts$, and let
\[
\mathcal U = \, {}^\st\{O\in \mathcal{B} : \mu \not\in O\}.
\]
This set is obtained by SC (Countable Standardisation, see
Definition~\ref{d45}) with a nonstandard parameter (namely,~$\mu$),
available in SCOT \cite[Lemma~3.1]{21e}.  By the above and
transfer,~$\mathcal U$ is a standard open cover of~$\ts$.  If~$\ts$
were compact,~$\mathcal U$ would have, by transfer, a standard finite
open subcover~$O_1,\ldots,O_k$.  Then~$\mu\in O_i$ for some~$1\leq
i\leq k$, contradicting the definition of~$\mathcal U$.
\end{proof}

%
%

Since second countable implies Lindel\"of,
%
%
we have the equivalence of the two definitions of compactness, for
second-countable spaces in SCOT.

\section{Continuity}
\label{s6}

Based on the results of Section~\ref{s3}, the following can be proved
conservatively over ZF+ADC using infinitesimals.  We will first
discuss continuity over SPOT.

In this section, $f$ is a standard map between standard topological spaces.  $f$ is
said to be S-\emph{continuous at}~$c$ if whenever~$x\simeq c$, one
has~$f(x)\simeq f(c)$.

\begin{lemma}
If a standard map~$f$ from a first countable topological space into a
topological space is S-continuous at a standard point~$c$ then~$f$ is
continuous at~$c$.
\end{lemma}

\begin{proof}
Let~$\mathcal{B}_c$ be a standard countable base of open neighborhoods
of~$c$.  Assume that~$f$ is not continuous at~$c$.  Then there is a
standard open neighborhood~$U$ of~$f(c)$ such that for every
(standard) finite~$O_1,\ldots,O_k \in \mathcal{B}_c$ there is~$x \in
\bigcap_{1\le i \le k} O_i$ with~$f(x) \notin U$.  By Countable
Idealization there is~$x$ such that~$x \in O$ holds for all
standard~$O \in \mathcal{B}_c$ and~$f(x) \notin U$. Then~$x \simeq c$
and~$f(x) \not \simeq f(c)$, a contradiction.
\end{proof}

\begin{lemma}
\label{l62b}
If a standard function~$f$ is continuous at a standard point~$c$
then~$f$ is S-continuous at~$c$.
\end{lemma}

\begin{proof}
Assume that~$f$ is not S-continuous at~$c$.  Then there is ~$x \simeq
c$ for which~$f(x) \not \simeq f(c)$, i.e.,~$f(x) \notin U$ holds for
some standard neighborhood~$U$ of~$f(c)$.  By continuity of~$f$ there
is a standard open neighborhood~$O$ of~$c$ such that~$z \in O$
implies~$f(z) \in U$.  As~$x \simeq c$, we have~$x \in O$ and
hence~$f(x) \in U$, a contradiction.
\end{proof}

\section{Continuous image of compacts}
\label{s7}

We prove the following well-known result in SCOT.

\begin{theorem}
Let~$f: \ts\to Y$ be a continuous map between second-countable
topological spaces.  Let~$E\subseteq \ts$ be compact.  Then $f(E)$ is
compact.
\end{theorem}

\begin{proof}
We prove the theorem under the assumption that  $f, T, Y$ are standard; its validity for arbitrary $ f, T. Y$ follows by transfer.
By
Lemma~\ref{l32}, every point~$x\in E$ is infinitely close to a
standard point~$p\in E$.  By the nonstandard characterisation of
continuity of~$f$ (Lemma~\ref{l62b}), the point~$f(x)$ is infinitely
close to~$f(p)$.  Thus every point~$f(x)$ in the image~$f(E)$ is
infinitely close to a standard point~$f(p)\in f(E)$.  By
Lemma~\ref{l31}, applied to~$f(E)$, the space~$f(E)$ is compact.
\end{proof}

The proof compares favorably with the traditional proof using
pullbacks of open covers, and is as effective (in the sense explained
in Section~\ref{s1}) as the traditional proof.

\section{Characterisation of uniform continuity}
\label{s8}

Let~$D,E$ be standard metric spaces; we will denote the distance
functions by~$|\cdot|$.  A standard map~$f\colon D\to E$ is
\emph{uniformly continuous} on~$D$ if
\begin{equation}
\label{e571z}
\begin{aligned}
(\forall\epsilon & \in\R^+) (\exists\delta\in\R^+) \\&
  \underline{(\forall x\in D) (\forall x'\in D)\;
    \big[|x'-x|<\delta \rightarrow |f(x')-f(x)|<\epsilon\big].}
\end{aligned}
\end{equation}

$f$ is \emph{S-continuous on}~$D$ if
\begin{equation}
\label{i2b}
(\forall x\in D)(\forall x'\in D)\left[x\simeq x'\;\rightarrow
\; f(x)\simeq f(x') \right].
\end{equation}

\begin{lemma}
If~$f$ is uniformly continuous on~$D$ then it is~$S$-continuous there.
\end{lemma}

\begin{proof}
To show that condition~\eqref{e571z} implies~\eqref{i2b}, fix a
standard parameter~$\epsilon$.  By downward transfer, there is a standard~$\delta$ such that  the \underline{underlined} part of
formula~\eqref{e571z} holds:
\begin{equation}
\label{e572}
(\forall x\in  D)(\forall x'\in
 D)\;\big[|x'-x|<\delta\rightarrow| f(x')- f(x)|<\epsilon\big].
\end{equation}
If~$x\simeq x'$ then the condition~$|x-x'|<\delta$ is satisfied
regardless of the value of the standard number~$\delta>0$.  Therefore
\begin{equation}
\label{e573}
  |f(x')- f(x)|<\epsilon.
\end{equation}
Since~\eqref{e573} is true for each standard~$\epsilon>0$, we conclude
that~$ f(x')\simeq f(x)$, proving \eqref{i2b}.
\end{proof}

\begin{lemma}
\label{l62}
If~$f$ is S-continuous on~$D$ then~$f$ is uniformly continuous
there.
\end{lemma}

\begin{proof}
We will show the contrapositive statement, namely that~$\neg
\eqref{e571z}$ implies~$\neg \eqref{i2b}$.  Assume the negation
of~\eqref{e571z}.  By downward transfer it follows that there exists a standard
number~$\epsilon>0$ such that
\begin{equation}
\label{e574}
(\forall \delta\in\R^+)(\exists  x\in D)(\exists x' \in
D)\;\big[|x'-x|<\delta \;\wedge\; |f(x')-f(x)|>\epsilon\big].
\end{equation}
 The formula is true
for all positive~$\delta$, so in particular it holds for an
infinitesimal~$\delta_0>0$.  For this value, we obtain
\begin{equation}
\label{e437}
(\exists x\in  D)(\exists x'\in D)\;\big[|x'-x|<\delta_0
\;\wedge\;| f(x')- f(x)|>\epsilon\big].
\end{equation}
Fix such~$x$ and~$x'$. The condition~$|x'-x|<\delta_0$ implies
that~$x\simeq x'$, while~$| f(x')- f(x)|>\epsilon$.  As the lower
bound~$\epsilon>0$ is standard, it follows that~$f(x')\not\simeq
f(x)$.  This violates condition \eqref{i2b} and establishes the
required contrapositive implication~$\neg \eqref{e571z} \implies \neg
\eqref{i2b}$.
\end{proof}

\section{Continuity implies uniform continuity}
\label{s9}

As shown in Section~\ref{s8}, the theory SPOT proves that uniform
continuity of a map between metric spaces amounts to S-continuity at
all points (standard and nonstandard) of the domain.

\begin{theorem}
A continuous map from a compact metric space to a metric space is
uniformly continuous.
\end{theorem}

\begin{proof}
Let $f: E \to Y$ where $f, E, Y$ are standard. By the characterisation of the compactness
of~$E$ (Lemma~\ref{l32}), if~$x\in E$ then~$x$ is infinitely close to
a standard point~$p\in E$.  For each ~$x'\simeq x$, one has~$x'\simeq
p\simeq x$.  If~$f$ is continuous at~$p$ then~$f(x')\simeq f(p)\simeq
f(x)$, and therefore~$f$ is S-continuous at all points of~$E$,
establishing uniform continuity by Lemma~\ref{l62}.
By transfer, the theorem holds for arbitrary $f, E, Y$.
\end{proof}

This proof in SPOT compares favorably with the traditional proof:
given~$\epsilon>0$, we need to find~$\delta>0$ such that if
$d_E(x,y)<\delta$ then one has~$d_Y(f(x),f(y))<\epsilon$.  By
continuity, for each~$x\in E$ there exists a~$\delta_x>0$ such that
if~$d(x,y)< \delta_x$ then~$d(f(x),f(y))<\frac\epsilon2$.
Then
\[
\left\{B(x,\tfrac{\delta_x}{2}): x\in E\right\}
\]
is an open cover of~$E$.  By compactness, there are points
$x_1,\ldots,x_n\in E$ such that~$\left\{B(x_1,\frac{\delta_1}2),
\ldots, B(x_n,\frac{\delta_n}2)\right\}$ is a finite subcover
covering~$E$.  Let~$\delta=\min(\frac{\delta_1}2, \ldots,
\frac{\delta_n}2)$.  If~$y,z\in E$ and~$d(y,z)<\delta\leq
\frac{\delta_k}2$ for each~$k=1,\ldots,n$, then by the triangle
inequality
\[
d(x_k,z) \leq d(x_k,y)+d(y,z) \leq \tfrac{\delta_k}2 +
\tfrac{\delta_k}2 < \delta_k.
\]
Therefore 
\[
d(f(y),f(z))\leq d(f(y),f(x_k))+ d(f(x_k), f(z)) < \tfrac\epsilon2 +
\tfrac\epsilon2 = \epsilon,
\]
establishing uniform continuity.

\section{Heine--Borel theorem}
\label{s10}

Here we present an effective approach to the Heine--Borel theorem
exploiting the characterisation of compactness of Section~\ref{s3}.  A
standard set~$C$ in a standard metric space is \emph{closed} if every
standard element near some element of~$C$ is actually in~$C$.  
It is \emph{bounded}%
\footnote{Here~$C$ must be a subset of a metric space for the notion
  of boundedness to make sense.  Note that Countable Idealisation (CI)
  is needed to prove its equivalence to the usual definition ``There
  is a real~$r$ such that for all~$x \in C$,~$d(x,p) < r$ holds for
  some (equivalently, all)~$p\in C$.''  CI is available in SPOT.}
if it contains no unlimited elements.  If standard~$x$ and~$y$ are
infinitely close then~$x=y$.

\begin{lemma}
If~$C$ is compact then it is closed and bounded.
\end{lemma}

\begin{proof}
Every element of~$C$ is nearstandard in~$C$ by compactness, so there
are no unlimited elements, i.e.,~$C$ is bounded. Let a standard
point~$x$ be infinitely close to some~$y\in C$. By compactness of~$C$
there is a standard~$z \in C$ such that~$ z\simeq y$.  Thus~$x = z\in
C$ and~$C$ is closed.
\end{proof}

\begin{lemma}
For standard~$n$, if~$C\subseteq \R^n$ is closed and bounded then it
is compact.
\end{lemma}

\begin{proof}
For standard~$n$, the condition of infinite proximity in~$\R^n$
amounts to the condition of infinite proximity for each of the~$n$
coordinates.  A similar remark applies to boundedness.

Let~$x\in C$.  Since~$C$ is bounded,~$x$ is limited and hence
infinitely close to a standard~$y\in\R^n$.  By closure,~$y\simeq x$
entails~$y\in C$.  Thus~$x$ is nearstandard in~$C$.  This proves
that~$C$ is compact.
\end{proof}

\section{Radically Elementary Probability Theory}
\label{s11}

Nelson's \emph{Radically Elementary Probability Theory} is based on
traditional mathematics plus axioms 1 through 5 stated in
\cite[Section~4, pp.\,13--14]{Ne87}.  The axioms 1 through 4 hold in
SPOT.\, Axiom~5, which according to Nelson is rarely used, is the
axiom~CC (see Definition~\ref{d44}).  It follows that \emph{Radically
  Elementary Probability Theory} is conservative over ZF+ADC.
Furthermore, it follows that all results from \cite{Ne87}
automatically hold in SCOT.\, In particular, this includes Nelson's
S-integral.

Further applications include proofs in SPOT of Peano and Osgood
theorems for ordinary differential equations \cite{23z}.

Our perspective fits with a relative view of the foundations of
mathematics such as that provided by Hamkins' multiverse.  For the
relation between the Gitman--Hamkins ``toy'' model of the multiverse
\cite{GH} and nonstandard analysis, see Fletcher et
al.~\cite[Section~7.3]{17f}.

\medskip
\noindent {\bf Acknowledgment}.  The authors wish to thank the referee
for his constructive critique of the first draft.  We are grateful to
Karl Kuhlemann for helpful comments.

\end{document}